\documentclass{birkjour}
 \newtheorem{thm}{Theorem}[section]
 
 \newtheorem{lem}[thm]{Lemma}
 \newtheorem{prop}[thm]{Proposition}
 \theoremstyle{definition}
 \newtheorem{defn}[thm]{Definition}
 \theoremstyle{remark}
 \newtheorem{rem}[thm]{Remark}
 
 \numberwithin{equation}{section}
 
 \def\Rn{{\mathbb{R}^n}}

\usepackage{amsmath,amscd,amsthm,amssymb}

\begin{document}

%
%
%
%
%
%
%
%
%

\title[Associated Generalized Classical Lorentz Space]
 {Associated Spaces of Generalized Classical Lorentz Spaces $G\Lambda_{p,\psi;\varphi}$}

\author[C. Aykol]{Canay Aykol}

\address{%
Ankara University\\
Department of Mathematics\\
06100 Tandogan\\
Ankara\\
Turkey
}

\email{aykol@science.ankara.edu.tr}

\author[A. Gogatishvili]{Amiran Gogatishvili}

\address{%
Institute of Mathematics\\
Academy of Sciences of the Czech Republic\\
Zitna 25, CZ 115 67 Praha 1\\
Czech Republic\\
}

\email{ gogatish@math.caz.cz}

\thanks{The research was supported by exchange program between Academy of Sciences of the Czech Republic and TUBITAK. C. Aykol was partially supported by The Turkish
Scientific and Technological Research Council (TUBITAK, programme 2211). The research of A. Gogatishvili was partially supported by the grant P201-13-14743S of the Grant Agency of the Czech Republic and RVO: 67985840, Shota Rustavely National Science Foundation grants no. 13/06 (Geometry of function spaces, interpolation and embeding theorems) and  no. 31/48 (Operators in some function spaces and their applications in Fourier Analysis). The research of V.S. Guliyev was supported by the grant of 2012-Ahi Evran University Scientific Research Projects (PYO-FEN 4001.13.18).}
\author[V. S. Guliyev]{Vagif S. Guliyev}
\address{(1) Ahi Evran University\\
Department of Mathematics\\
Bagbasi campus\\
Kirsehir\\
Turkey}
\address{(2) Institute of Mathematics and Mechanics\\
of NAS of Azerbaijan 9, B.Vaxabzade\\
Baku\\
Azerbaijan Republic\\}
\email{vagif@guliyev.com}



\date{October 3, 2013}

\begin{abstract}
In this paper we have calculated the associate norms of the
$G\Lambda_{p,\psi;\varphi}$  generalized classical Lorentz spaces.
\end{abstract}

\maketitle

\

\section{Introduction}

This paper aims at calculating the associate norm of the generalized classical Lorentz spaces.
We have used the characterization of the weighted reverse Hardy inequality to calculate the associate norm of the spaces $G\Lambda_{p,\psi;\varphi}$.

\

\section{Definitions and Preliminary Tools}

Let $E$ be a measurable subset of $\Rn$.
We denote by $L_{p,E}$ the class of all measurable functions $f$ defined on $E$ for which
\begin{equation*}
\|f\|_{L_{p,E}}:=\left(\int_{E}|f(y)|^{p}dy\right)^{\frac{1}{p}}<\infty,0<p<\infty,
\end{equation*}
\begin{equation*}
\|f\|_{L_{\infty,E}}:=\sup\{\alpha:|\{y\in E:|f(y)|\geq \alpha\}|>0\},
\end{equation*}
and denote by $WL_{p}$ the weak $L_{p}$ space such that
\begin{equation*}
\|f\|_{WL_{p,E}}:=\{f:\sup_{\alpha>0}\alpha\mu_{f}(\alpha)^{1/p}<\infty\},
\end{equation*}
where $\mu_{f}(\alpha)$ denotes the distribution
function of $ f $ given by
\begin{equation*}
\mu_{f}(\alpha)=\mu\{x\in E
:|f(x)|>\alpha \}.
\end{equation*}

%

\

Let for $0<p, q \le \infty$ and $r>0$
\begin{equation}\label{FBGkb}
\|f\|_{p,q;(0,r)}:=\| t^{\frac{1}{p}-\frac{1}{q}} f^{\ast}(t) \|_{q, (0,r)},
\end{equation}
where $f^{\ast}$ is the decreasing rearrangement of $f$  defined by
\begin{equation*}
f^{\ast}(t)=\inf \;\{\lambda >0~:~\mu_{f}(\lambda)\leq t\},\
\ \forall t\in (0,\infty).
\end{equation*}

\

We denote by $\mathfrak{M}(\Rn,\mu)$ be the set of all extended real valued $\mu$-measurable functions on $\Rn$  and $\mathfrak{M}^{+}(0,\infty)$ the set of all non-negative measurable functions on $(0,\infty)$,
$ \mathfrak{M}^{+}(0,\infty ;\uparrow) $ the set of all non-decreasing functions from $\mathfrak{M}^{+}(0,\infty) $.

\

Now we recall  definitions of Lorentz, classical Lorentz and generalized classical Lorentz spaces.

\

\begin{defn}\label{cay2}
The Lorentz space $L_{p,q} \equiv L_{p,q}(\Rn)$, $0<p, q \le \infty$ is the collection of all measurable functions $f$ on $\Rn $ such the quantity
\begin{equation}\label{ayl1}
\|f\|_{p,q}:=\|f\|_{p,q;(0,\infty)}=\| t^{\frac{1}{p}-\frac{1}{q}} f^{\ast}(t) \|_{q, (0,\infty)}
\end{equation}
is finite.
\end{defn}
Note that $L_{p,\infty}(\Rn)=WL_{p}(\Rn)$ (see, for example, \cite{St1}).

\

Note that, $L_{p,p}=L_p$ for $0<p \le \infty$.

\

If $1\leq q\leq p$ or $p=q=\infty$, then the functional $\|f\|_{p,q}$ is a norm.

\

For $0<q \le r \le \infty$ we have, with continuous embeddings, that

$$
L_{p,q} \subset L_{p,r}.
$$

\

The function $f^{\ast\ast }: (0,\infty)\rightarrow[0,\infty]$ is defined
as
\begin{equation*}
f^{\ast
\ast}(t)=\frac{1}{t}\int_{0}^{t}f^{\ast}(s)ds.
\end{equation*}

\

In the case $ 0<p,q \le \infty $, we give a functional $\|\cdot\|_{p,q}^{\ast}$ by
\begin{equation*}
\|f\|_{p,q}^{\ast}:=\|f\|_{p,q;(0,\infty)}^{\ast}=\| t^{\frac{1}{p}-\frac{1}{q}} f^{\ast \ast}(t) \|_{q, (0,\infty)}
\end{equation*}
(with the usual modification if $0<p\leq\infty$, $q=\infty$) which is a norm on $ L_{p,q}(\Rn)$ for $1<p<\infty$, $1\leq q \leq\infty$ or $p=q=\infty$.

\

If $1<p\leq\infty,$ $1\leq q \leq \infty$, then
\begin{equation}\label{cay}
\|f\| _{p,q}\leq \|f\|_{p,q}^{\ast}
\leq \frac{p}{p-1}\|f\|_{p,q}.
\end{equation}

About $L_{p,q}(\Rn)$ Lorentz spaces see \cite{Hunt66, St1}.

\

\begin{defn}
Let $0<p,q\leq\infty$ and $\psi\in \mathfrak{M}^{+}(0,\infty)$. We denote by $\Lambda_{p,\psi}(\Rn)$
the classical Lorentz spaces, the spaces of all measurable functions with finite quasinorm
\begin{equation*}
\Lambda_{p,\psi}(\Rn):=\{f \in \mathfrak{M}(\Rn): \|f\|_{\Lambda_{p,\psi}} :=\|\psi f^{*}\|_{p,(0,\infty)}\}.
\end{equation*}
\end{defn}
Therefore the following statement
$$ \Lambda_{p,t^{\frac{1}{p}-\frac{1}{q}}}(\Rn) = L_{p,q}(\Rn)$$
is valid.


The spaces $\Lambda^{p}(w) \equiv \Lambda_{p,w^{1/p}}$ were introduced by Lorentz in 1951 in \cite{Lorentz2}. Spaces whose norms involve $f^{**}$ appeared explicitly for the first time in Calderon's paper \cite{Calderon1964}. In \cite{Sawyer1990} Sawyer
give description of the dual of $\Lambda^{p}(w)$.

Lorentz \cite{Lorentz2} proved that, for $p \ge 1$, $\|f\|_{\Lambda^{p}(w)}$ is a norm if and only if $w$ is nonincreasing. The class of weights for which $\|f\|_{\Lambda^{p}(w)}$ is merely equivalent to a Banach
norm is however considerably larger. In fact it consists of all those weights w which, for some $C$ and all $t > 0$, satisfy
\begin{align*}
t^p \int_t^{\infty} x^{-p} w(x) dx \le C \int_0^t  w(x) dx
~~~~ \mbox{when} ~~~~ p \in (1,\infty)
\end{align*}
(\cite{Sawyer1990}, Theorem 4], see also \cite{ArinoMuck1990}), or
\begin{align*}
\frac{1}{t} \int_0^t w(x) dx \le \frac{C}{s} \int_0^s  w(x) dx ~~~~ \mbox{for}~~  0 < s \le t ~~~~ \mbox{when} ~~ p=1
\end{align*}
(\cite{CGS1996}, Theorem 2.3).

In \cite{Haaker}, Theorem 1.1 (see also \cite{CarroS1993}, Corollary 2.2, \cite{KamMalig}, p. 6) it was observed that the functional $\|f\|_{\Lambda^{p}(w)}$; $0 < p \le \infty$, does not have to be a quasinorm. It was shown
that it is a quasinorm if and only if the function $W(t)=\int_0^t w(s)ds$
satisfies the $\Delta_2$-condition, i.e.,
\begin{align*}
W(2t) \le C W(t) ~~~~ \mbox{for some} ~~ C>1 ~~ t \in (0,\infty).
\end{align*}
In \cite{CKMP2004} were given necessary and sufficient conditions for $\Lambda^{p}(w)$ to be a linear space.
About historical developments classical Lorentz spaces see \cite{CPSS2001}.

\begin{defn} 
Let $0<p,q\leq\infty$ and let $\varphi,\psi\in \mathfrak{M}^{+}(0,\infty)$. We denote by $GL_{p,q;\varphi}(\Rn)$
the generalized Lorentz spaces, the spaces of all measurable functions with finite quasinorm
\begin{equation*}
GL_{p,q;\varphi}(\Rn):=\{f \in \mathfrak{M}(\Rn): \|f\|_{GL_{p,q;\varphi}}
:=\sup_{ r>0}\varphi(r) \|t^{\frac{1}{p}-\frac{1}{q}}f^{*}(t)\|_{q;(0,r)}\}.
\end{equation*}
\end{defn}

\begin{defn} 
Let $0<p\leq\infty$ and let $\varphi,\psi\in \mathfrak{M}^{+}(0,\infty)$. We denote by $G\Lambda_{p,\psi;\varphi}(\Rn)$ the generalized classical Lorentz spaces, the spaces of all measurable functions with finite quasinorm
\begin{equation*}
G\Lambda_{p,\psi;\varphi}(\Rn):=\{f \in \mathfrak{M}(\Rn): \|f\|_{G\Lambda_{p,\psi;\varphi}}
:=\sup_{ r>0}\varphi(r) \|\psi(\cdot)f^{*}(\cdot)\|_{p;(0,r)}\}.
\end{equation*}
\end{defn}

Therefore the following statement
$$
L_{p,q}=GL_{p,q;1}, ~~~~ L_{p;\varphi}=GL_{p,p;\varphi}, ~~~~ L_{p}=GL_{p,p;1},
$$
$$
\Lambda_{p,\psi}=G\Lambda_{p,\psi;1}, ~~~~ L_{p;\varphi}=G\Lambda_{p,1;\varphi}, ~~~~ L_{p}=G\Lambda_{p,1;1}
$$
is valid. Note that, the space
$$L_{p;\varphi}(\Rn)=\{ f \in \mathfrak{M}(\Rn): \, \|f\|_{L_{p;\varphi}} :=\sup_{r>0}\varphi(r) \|f^{*}\|_{p;(0,r)}<\infty \}$$
we called generalized Lebesgue spaces.

Recall the definition of generalized Marcinkiewicz space

\[ M_{p;\varphi}(\Rn)=\{ f \in \mathfrak{M}(\Rn): \,\|f\|_{M_{p,\varphi}}:=\sup_{E}\varphi(|E|) \|f\|_{p,E}<\infty \}, \]
where the supremum is taken for all measurable subset of $\Rn$.
It easy to see that

\[ \sup_{E}\varphi(|E|) \|f\|_{p,E}= \sup_{t>0}\varphi(t) \|f\|_{p,(0,t)}
\]
Therefore, we have $G\Lambda_{p,1;\varphi}(\Rn)=M_{p;\varphi}(\Rn)$ and $G\Lambda_{1,1;\varphi}(\Rn)=M_{\varphi}(\Rn)$.

\begin{rem} \label{geyd1}
Note that, the classical Lorentz space $\Lambda_{p,\psi}(\Rn)$ is not a linear space, therefore the generalized classical Lorentz spaces $G\Lambda_{p,\psi;\varphi}(\Rn)$
are also not linear spaces in general. It is easy to see that the condition
\begin{align*}
\int_0^{2t} \psi^{p}(x) dx \le C \int_0^t  \psi^{p}(x) dx ~~~~ \mbox{for some} ~~ C>1 ~~ t \in (0,\infty)
\end{align*}
is sufficient for the generalized classical Lorentz spaces $G\Lambda_{p,\psi;\varphi}(\Rn)$ to be a quasinorm space.

Also the condition
\begin{align*}
t^p \int_t^{\infty} x^{-p} \psi^{p}(x) dx \le C \int_0^t  \psi^{p}(x) dx
~~~~ \mbox{when} ~~~~ p \in (1,\infty)
\end{align*}
or
\begin{align*}
\frac{1}{t} \int_0^t \psi(x) dx \le \frac{C}{s} \int_0^s  \psi(x) dx ~~~~ \mbox{for}~~  0 < s \le t ~~~~ \mbox{when} ~~ p=1
\end{align*}
is sufficient for the generalized classical Lorentz spaces $G\Lambda_{p,\psi;\varphi}(\Rn)$ to be a norm space.
\end{rem}
\begin{rem}
The problems mentioned in the Remark \ref{geyd1} are not as easy as the problems in the case of the classical Lorentz spaces. At this point it's required to have the characterizations of the embedding between these generalized classical Lorentz spaces and the boundedness of the maximal operator in these spaces. We  haven't got the solution of these problems yet. To solve these problems it will be useful to find the characterizations of the associate spaces of the generalized classical Lorentz spaces. In this paper we give the characterizations of the associate spaces of these spaces.
Note that, in the special case $\varphi(t)=t^{-\frac{\lambda}{q}}$ the generalized Lorentz space $GL_{p,q;\varphi}(\Rn)$ was introduced and investigated in \cite{AykGulSer}.
\end{rem}

\begin{defn}
Let $X$ be a set of functions from $\mathfrak{M}(\mathcal{R},\mu)$
endowed with a positively homogenous functional $\|\cdot\|_{X}$ defined for every
$f\in \mathfrak{M}(\mathcal{R},\mu)$ and such that $f\in X$ if and only if $\|f\|_{X}<\infty$,
we define the associate space $X'$ of $X$ as the set of all functions $f\in \mathfrak{M}(\mathcal{R},\mu)$ such that
$\|f\|_{X'}<\infty$, where
\begin{equation*}
\|f\|_{X'}=\sup\{\int_{X}fgd\mu:\|g\|_{X}\leq 1\}
\end{equation*}
In what follows we assume $\mathcal{R}=\Rn$ and $d\mu=dx$.
\end{defn}

\begin{prop}\label{prp1}(\cite{BS}, p. 58) Let $\|\cdot\|$ be a rearrangement-invariant
function norm over a resonant measure space $(X,d\mu)$. Then the associate norm $\|\cdot\|_{X'}$
is also rearrangement invariant. Furthermore,
\begin{align*}
\|f\|_{X'}&=\sup\{\int_{X}fgd\mu:\|g\|_{X}\leq 1\}
\\
&=\sup\{\int_{0}^{\infty}f^{*}(t)g^{*}(t)dt:\|g\|_{X}\leq 1\}
\end{align*}
holds.
\end{prop}

Throughout the paper, we write $A \lesssim B$ if there exists a positive constant
$C$, independent of appropriate quantities such as functions, satisfying $A \leq CB$.
We write $A \approx B$ when $A \lesssim B$ and $B \lesssim A$.

\section{Reverse Hardy Inequality}
Let us recall some results from \cite{Gog}.


\

Let $u,v$ and $w$ will denote weights, that is,
locally integrable non-negative functions on $(0,\infty)$. We set, once and for all
\begin{equation*}
U(t)=\int_{0}^{t}u(s)ds, V(t)=\int_{0}^{t}v(s)ds, W(t)=\int_{0}^{t}w(s)ds.
\end{equation*}

We assume that $U(t)>0$ for every $t\in(0,\infty)$. We then denote
\begin{equation*}
f_{u}^{**}(t)=\frac{1}{U(t)}\int_{0}^{t}f^{*}(s)u(s)ds, t\in(0,\infty).
\end{equation*}
When $u\equiv 1$ (hence $U(t) = t$), we will omit the subscript u.

\

Furthermore, for given $q\in(0,\infty)$ and every $f\in \mathfrak{M}(\mathcal{R},\mu)$,
the necessary and sufficient conditions for the inequality
\begin{equation}\label{cay1}
\left(\int_{0}^{\infty}f^{*}(t)^{q}w(t)dt\right)^{\frac{1}{q}}
\lesssim ess\sup_{t\in(0,\infty)}f_{u}^{**}(t)v(t)
\end{equation}
were established in \cite{Gog}.

\

\begin{defn}

Let $\theta$ be a continuous strictly increasing function on $[0,\infty)$ such that
$\theta(0)=0$ and $\lim_{t\rightarrow \infty}\theta(t)=\infty$. Then we say $\theta$ is admissible.

\

Let $\theta$ be an admissible function. We say that a function $h$ is $\theta$-quasiconcave if $h$
is equivalent to a non-decreasing function on $[0,\infty)$ and $\frac{h}{\theta}$ is equivalent
to a non-increasing function on $(0,\infty)$.
We say that a $\theta$-quasiconcave function $h$ is non-degenerate if
\begin{equation*}
\lim_{t\rightarrow 0^{+}}h(t)=\lim_{t\rightarrow \infty}\frac{1}{h(t)}
=\lim_{t\rightarrow \infty}\frac{h(t)}{\theta(t)}=\lim_{t\rightarrow 0^{+}}\frac{\theta(t)}{h(t)}=0.
\end{equation*}
The family of non-degenerate $\theta$-quasiconcave functions will be denoted by $\Omega _{\theta}$.
\end{defn}
\begin{lem}Let $u,v$ be weights above and let $\sigma$ defined by
\begin{equation}\label{cay4}
\sigma(t):=ess\sup_{s\in (0,t)}U(s)ess\sup_{\tau\in (s,\infty)}
\frac{v(\tau)}{U(\tau)},t\in (0,\infty)
\end{equation}
then,
\begin{equation*}
\sigma(t)\approx ess\sup_{s\in (0,\infty)}v(s)\frac{U(t)}{U(s)+U(t)}.
\end{equation*}
\end{lem}
\begin{defn}
Let $\sigma$ be an admissible function and let $\nu$ be a non-negative Borel measure on $[0,\infty)$.
We say that the function $h$ defined as
\begin{equation*}
h(t):=\sigma(t)\int_{[0,\infty)}\frac{d\nu(s)}{\sigma(s)+\sigma(t)}, t\in(0,\infty),
\end{equation*}
fundamental function of the measure $\nu$ with respect to $\sigma$. We will also say that the function $\nu$
is a representation measure of $h$ with respect to $\sigma$.

\

We say that $\nu$ is non-degenerate if the following conditions are satisfied for every $t\in(0,\infty)$:
\begin{equation*}
\int_{[0,\infty)}\frac{d\nu(s)}{\sigma(s)+\sigma(t)}<\infty,
\int_{[0,1]}\frac{d\nu(s)}{\sigma(s)}= \int_{[1,\infty)}d\nu(s)=\infty.
\end{equation*}
\end{defn}

\begin{rem}
Let $\sigma$ be an admissible function and let $\nu$ be a non-negative non-degenerate Borel measure on $[0,\infty)$.
Let $h$ be the fundamental function of $\nu$ with respect to $\sigma$. Then
\begin{equation*}
h(t)\approx\int_{0}^{t}\int_{[s,\infty)}\frac{d\nu(y)}{\sigma(y)}d\sigma(s), t\in(0,\infty),
\end{equation*}
and also
\begin{equation*}
h(t)\approx \int_{[0,t]}d\nu(s)+\sigma(t)\int_{[t,\infty]}\frac{d\nu(s)}{\sigma(s)}, t\in(0,\infty).
\end{equation*}
\end{rem}

\begin{thm}\label{mak1} \cite{Gog}
Let $q\in (0,\infty)$ and let $u,v,w$ be weights.
Assume that $u$ is such that $U$ is admissible. Let $\sigma$
defined by \eqref{cay4}, be non-degenerate with respect to $U$.
Let $\nu$ be the representation measure of $U^{q}/\sigma^{q}$ with respect to
$U^{q}$.

\

(i) If $1\leq q < \infty$, then \eqref{cay1} holds for all $f$ if and only if
\begin{equation*}
A(1)=\left(\int_{0}^{\infty}\sup_{s\in(t,\infty)}\frac{W(s)}{U(s)^{q}}d\nu(t)\right)^{\frac{1}{q}}<\infty.
\end{equation*}
Moreover, the optimal constant $C$ in \eqref{cay1} satisfies $C\approx A(1)$.

\

(ii) If $0<q<1$, then \eqref{cay1} holds for all $f$ if and only if
\begin{equation*}
A(2)=\left(\int_{0}^{\infty}\frac{\zeta(t)}{U(t)^{q}}d\nu(t)\right)^{\frac{1}{q}}<\infty,
\end{equation*}
where
\begin{equation*}
\zeta(t)=W(t)+U(t)^{q}\left(\int_{t}^{\infty}\left(\frac{W(s)}{U(s)}\right)^{\frac{q}{1-q}}w(s)ds\right)^{1-q}, ~~ t\in (0,\infty).
\end{equation*}
Moreover, the optimal constant in \eqref{cay1} satisfies $C\approx A(2)$.
\end{thm}

\section{Associated spaces of generalized classical Lorentz spaces $G\Lambda_{p,\psi;\varphi}$}

The associated spaces of classical Lorentz spaces $\Lambda_{p,\psi}$ was calculated in \cite{MarCarRap}.
\begin{thm}
Let $0 < p <\infty$, $\psi \in M^{+}(0,\infty)$. Then the associate spaces of $\Lambda_{p,\psi}$ are described as follows: \\
(i) If $0<p \le 1$, then
$$
\|f\|_{(\Lambda_{p,\psi})'} = \sup\limits_{t>0} \frac{tf^{**}(t)}{\Psi_p(t)},
$$
where $\Psi_p(t)=\|\Psi\|_{p,(0,t)}$. \\
(ii) If $1<p < \infty$, then
$$
\|f\|_{(\Lambda_{p,\psi})'} = \int_{0}^{\infty} \left(\frac{tf^{**}(t)}{\Psi_p(t)^p}\right)^{p'} \psi^p(t)dt.
$$
\end{thm}

In this section by using results of previous section we calculate the associated spaces
of generalized classical Lorentz spaces.
\begin{thm} \label{abdfk}
Let $0 < p <\infty$, $\psi \in M^{+}(0,\infty)$, $\varphi\in M^{+}(0,\infty,\downarrow)$ and $\varphi(r)r^{\frac{1}{p}}\in M^{+}(0,\infty,\uparrow)$. Then
the associate spaces of $G\Lambda_{p,\psi;\varphi}$ are described as follows: \\
(i) If $0<p \le 1$, then
$$
\|f\|_{(G\Lambda_{p,\psi;\varphi})'} = \int_{0}^{\infty}\sup\limits_{s\in (t,\infty)}\frac{sf^{**}(s)}{\Psi_p(s)}
d\nu(t),
$$
where $\nu$ is the representation measure of $\frac{1}{\varphi(t)}$ with respect to $\|\psi\|_{p,(0,t)}$. \\
(ii) If $1<p < \infty$, then
$$
\|f\|_{(G\Lambda_{p,\psi;\varphi})'} = \int_{0}^{\infty}\left(\int_{t}^{\infty}\left(\frac{sf^{**}(s)}{\Psi_p(s)}\right)^{p'} \psi^p(s)ds\right)^{1/p'}d\nu(t),
$$
where $\nu$ is the representation measure of $\frac{1}{\varphi(t)}$ with respect to $\Psi_p(t)$.
\end{thm}
\begin{proof} From Proposition \ref{prp1} we have
\begin{align}\label{cay3}
\|f\|_{(G\Lambda_{p,\psi;\varphi})'}&:=\sup\limits_{g\geq 0}
\frac{\int_{0}^{\infty}f^{*}(t)g^{*}(t)dt}
{\sup\limits_{r>0}\varphi(r)\left(\int_{0}^{r}g^{*}(t)^{p}\psi^p(t)dt\right)^{1/p}}.
\end{align}
If we take $g^{*}(t)=h^{*}(t)^{\frac{1}{p}}$, then we can write
\begin{align*}
\|f\|_{(G\Lambda_{p,\psi;\varphi})'}=\left[\sup\limits_{h\geq 0}\frac{\left(\int_{0}^{\infty}f^{*}(t)h^{*}(t)^{\frac{1}{p}}dt\right)^{p}}
{\sup\limits_{r>0}\varphi(r)^{p}\int_{0}^{r}h^{*}(t)\psi^p(t)dt}\right]^{1/p}.
\end{align*}
If we define the function $u(t)=\psi^p(t)$, then we get
\begin{align*}
h_{u}^{**}(t)&=\frac{1}{\Psi_p^p(t)}\int_{0}^{t}h^{*}(s)\psi^p(s)ds.
\end{align*}
Therefore
\begin{align}\label{cny}
\left[\sup\limits_{h\geq 0}\frac{\left(\int_{0}^{\infty}f^{*}(t)h^{*}(t)^{\frac{1}{p}}dt\right)^{p}}
{\sup\limits_{r>0}\varphi(r)^{p}\int_{0}^{r}h^{*}(t)\psi^p(t)dt}\right]^{1/p}
&\approx\left[\sup\limits_{h\geq 0}\frac{\left(\int_{0}^{\infty}f^{*}(t)h^{*}(t)^{\frac{1}{p}}dt\right)^{p}}
{\sup\limits_{r>0}\varphi(r)^{p}\Psi_p^p(r)h_{u}^{**}(r)}\right]^{1/p}.
\end{align}

\

(i) Let $1<p<\infty$.

\

In Theorem \ref{mak1} if we take $q=\frac{1}{p}$, $w(t)=f^{*}(t)$, $U(t)=\Psi_p^p(t)$,
$v(t)=\varphi(t)^{p}\Psi_p^p(t)$ and $\nu$ be the representation measure of $\frac{1}{\varphi}$
with respect to $\Psi_p(t)$, which means
\begin{equation*}
\frac{1}{\varphi(t)}\approx\int_{0}^{t}d\nu(s)+\Psi_p(t)\int_{t}^{\infty}\frac{d\nu(s)}{\Psi_p(s)}.
\end{equation*}

Then we get

\begin{equation*}
RHS\eqref {cny}\approx \int_{0}^{\infty}\frac{\zeta(t)}{\Psi_p(t)}d\nu(t),
\end{equation*}
where
\begin{equation*}
\zeta(t)=tf^{**}(t)+\Psi_p(t)\left(\int_{t}^{\infty}
\left(\frac{sf^{**}(s)}{\Psi_p^p(s)} \right)^{\frac{1}{p-1}}f^{*}(s)ds\right)^{1/p'},~~ t\in (0,\infty).
\end{equation*}

\

Furthermore, we have
\begin{align} \label{gog11}
\zeta(t)&\approx \Psi_p(t)\left(\int_{t}^{\infty}
\left(\frac{sf^{**}(s)}{\Psi_p^p(s)} \right)^{p'}\psi^p(s)ds\right)^{1/p'}
 =: \zeta_{1}(t).
\end{align}
Clearly
\begin{align} \label{sdgk1}
tf^{**}(t)&=(p-1)^{\frac{1}{p'}}tf^{**}(t)\Psi_p(t)
\left(\int_{t}^{\infty}\frac{\psi^p(s)}{\Psi_p(s)^{pp'}}ds\right)^{1/p'} \notag
\\
& \lesssim \Psi_p(t)\left(\int_{t}^{\infty}
\left(\frac{sf^{**}(s)}{\Psi_p^p(s)}\right)^{p'}\psi^p(s)ds\right)^{1/p'}=\zeta_{1}(t).
\end{align}

Also by partial integration
\begin{align} \label{avfd1}
&\int_{t}^{\infty}
\left(\frac{sf^{**}(s)}{\Psi_p^p(s)} \right)^{\frac{1}{p-1}}f^{*}(s)ds  \notag
\\
&=\frac{(sf^{**}(s))^{p'}}{\Psi_p^p(s)^{\frac{1}{p-1}}}|_{t}^{\infty}
-\int_{t}^{\infty}(sf^{**}(s))^{p'}d(\Psi_p^p(s))^{-\frac{1}{p-1}} \notag
\\
&=-\left(tf^{**}(t)\right)^{p'}(\Psi_p^p(t))^{-\frac{1}{p-1}}+
\frac{1}{p-1}\int_{t}^{\infty}\left(\frac{sf^{**}(s)}{\Psi_p^p(s)}\right)^{p'}\psi^p(s)ds.
\end{align}
Then from \eqref{sdgk1} and \eqref{avfd1} we get
\begin{equation}\label{asde2}
\zeta(t) \lesssim \zeta_{1}(t).
\end{equation}

Furthermore, from \eqref{avfd1} we also have
\begin{align*}
\zeta_{1}(t)\leq \Psi_p(t)
\left(\Psi_p(t)^{-\frac{p}{p-1}}(tf^{**}(t))^{p'}
+\int_{t}^{\infty}
\left(\frac{sf^{**}(s)}{\Psi_p^p(s)}\right)^{\frac{1}{p-1}}f^{*}(s)ds\right)^{1/p'}=\zeta(t).
\end{align*}
Therefore together with \eqref{asde2} we get \eqref{gog11} and consequently we have
\begin{align*}
RHS\eqref {cny}\approx\int_{0}^{\infty}\left(\int_{t}^{\infty}\left(\frac{sf^{**}(s)}{\Psi_p^p(s)}\right)^{p'}
\psi^p(s)ds\right)^{1/p'}d\nu(t).
\end{align*}

\

(ii) Let $0<p\leq 1$.  If we take $q=\frac{1}{p}$ in Theorem \ref{mak1}, then $1\leq q <\infty$ and we obtain
\begin{align*}
RHS\eqref {cny}\approx \int_{0}^{\infty}\sup\limits_{s\in (t,\infty)}\frac{sf^{**}(s)}{\Psi_p(s)}d\nu(t).
\end{align*}
\end{proof}

From the Theorem \ref{abdfk} we get the following theorem.

\begin{thm}
Let $0 < p,q <\infty$, $\psi, w \in M^{+}(0,\infty)$, $\varphi\in M^{+}(0,\infty,\downarrow)$ and $\varphi(r)r^{\frac{1}{p}}\in M^{+}(0,\infty,\uparrow)$. Then the following embedding
\begin{equation} \label{wklt}
G\Lambda_{p,\psi;\varphi} \hookrightarrow \Lambda_{q,w}
\end{equation}
is valid iff

i) $0 < p \le q <\infty$
$$
\int_{0}^{\infty}\sup\limits_{s\in (t,\infty)}\frac{\int_0^t w^q(s)ds}{\Psi_{\frac{p}{q}}(s)} d\nu(t)<\infty
$$
where $\nu$ is the representation measure of $\frac{1}{\varphi(t)}$ with respect to $\Psi_{\frac{p}{q}}(t)=\|\psi^q\|_{\frac{p}{q},(0,t)}$. \\
(ii) If $0<q<p < \infty$, then
$$
\int_{0}^{\infty}\left(\int_{t}^{\infty}\left(\frac{\int_0^t w^q(s)ds}{\Psi_{\frac{p}{q}}(s)}\right)^{\big(\frac{p}{q}\big)'} \psi^p(s)ds\right)^{1/\big(\frac{p}{q}\big)'}d\nu(t),
$$
where $\nu$ is the representation measure of $\frac{1}{\varphi(t)}$ with respect to $\Psi_{\frac{p}{q}}(t)$.
\end{thm}
\begin{proof}
Using a simple observation function, it is obvious that $f$ is decreasing if and only if, $f^q$ is decreasing for all $q>0$. We see that the embedding \eqref{wklt} holds if and only if the following embedding holds
\begin{equation} \label{wkltB}
G\Lambda_{\frac{p}{q},\psi^q;\varphi^q} \hookrightarrow \Lambda_{1,w^q}.
\end{equation}
One can get the required result by using Theorem \ref{abdfk}.
\end{proof}


\subsection*{Acknowledgment}
This paper has been written during the visit of A. Gogatishvili to Ankara University in Ankara and to Ahi-Evran University in Kirsehir.
A special gratitude he gives to these universities for their warm hospitality.

\


\begin{thebibliography}{1}

\bibitem{ArinoMuck1990} M. Arino and B. Muckenhoupt, \textit{Maximal functions on classical Lorentz spaces and Hardy's 
inequality with weights for nonincreasing functions}, Trans. Amer. Math. Soc. \textbf{320} (1990), 727-735.

\bibitem{AykGulSer} C. Aykol, V.S. Guliyev, A. Serbetci,
\textit{Local Morrey-Lorentz spaces and the boundedness of maximal operator in these spaces}, J. Inequal. Appl. \textbf{2013}, 2013:346.

\bibitem{BS} C. Bennett and R. Sharpley, \textit{Interpolation of Operators}, Academic Press, 1988.

\bibitem{Calderon1964} A. P. Calderon, \textit{Intermediate spaces and interpolation, the complex method}, Studia Math.
\textbf{24} (1964), 113-190.

\bibitem{CGS1996} M. J. Carro, A. Garcia del Amo and J. Soria, \textit{Weak-type weights and normable Lorentz spaces}, Proc. Amer. Math. Soc. \textbf{124} (1996), 849-857.

\bibitem{CPSS2001} M. J. Carro, L. Pick, J. Soria and V. D. Stepanov, \textit{On embeddings between classical Lorentz
spaces}, Math. Ineq. Appl. \textbf{4} (2001), 397-428.

\bibitem{CarroS1993} M. J. Carro and J. Soria, \textit{Weighted Lorentz spaces and the Hardy operator}, J. Funct. Anal. \textbf{112} (1993), 480-494.

\bibitem{CKMP2004} M. Cwikel, A. Kaminska, L. Maligranda and L. Pick, \textit{Are generalized Lorentz "spaces'' really spaces?} (English summary)
Proc. Amer. Math. Soc. \textbf{132}(12), (2004), 3615-3625.

\bibitem{Gog} A. Gogatishvili and L. Pick, \textit{Embeddings and Duality Theorems for Weak Classical Lorentz Spaces},
Canad. Math. Bull. Vol. \textbf{49} (1) (2006),82-95.

\bibitem{Haaker} A. Haaker, \textit{On the conjugate space of Lorentz space}, Technical Report, Lund 1970, 1-23.

\bibitem{Hunt66} R.A. Hunt, \textit{On $L(p, q)$ spaces}, Enseign. Math.  \textbf{12} (1966), 249-276.

\bibitem{KamMalig} A. Kaminska and L. Maligranda, \textit{Order convexity and concavity in Lorentz spaces $\Lambda^{p}(w)$, $0 < p < \infty$},
Studia Math. \textbf{160} (2004), 267-286.

\bibitem{MarCarRap} M. J. Carro, J. A. Raposo, J. Soria \textit{Recent Developments
in the Theory of Lorentz Spaces and Weighted Inequalities}, Memoirs of the Amer. Math. Soc., 2007.


\bibitem{Lorentz2} G.G. Lorentz, \textit{On the theory of spaces}, Pac. J. Math. \textbf{1} (1951), 411-429.

\bibitem{Sawyer1990} E. Sawyer, \textit{Boundedness of classical operators on classical Lorentz spaces},
Studia Math. \textbf{96} (2), (1990), 145-158.

\bibitem{St1}
E.M Stein and G. Weiss, \emph{Introduction to Fourier Analysis on
Euclidean Spaces}. Princeton Univ. Press 1971.

\end{thebibliography}
\end{document}